\documentclass[dvipdfmx]{amsart}
\usepackage{amsmath, amsfonts, amssymb,amscd,eucal,setspace}
\usepackage{latexsym}
\usepackage{color,graphicx,hyperref}

\usepackage{url}
\usepackage{amscd}
\usepackage{here}
\numberwithin{equation}{section}

\theoremstyle{plain}
\newtheorem{theo}{Theorem}[section]
\newtheorem{lemm}[theo]{Lemma}
\newtheorem{coro}[theo]{Corollary}
\newtheorem{prop}[theo]{Proposition}

\theoremstyle{definition}
\newtheorem{rema}{Remark}[section]
\newtheorem{exam}{Example}[section]

\def\C{\mathbb C}

\def\Z{\mathbb Z}

\DeclareMathOperator{\Hess}{Hess}

\DeclareMathOperator{\flag}{Fl}
\DeclareMathOperator{\Poin}{Poin}

\def\Sn{\mathfrak{S}_n}

\definecolor{gray7}{gray}{0.7}

\def\L{\textsf{\upshape L}}
\def\J{\bot}
\def\D{\Lambda}

\makeatletter
\@namedef{subjclassname@2020}{%
  \textup{2020} Mathematics Subject Classification}
\makeatother

\begin{document}
\title[The second cohomology]{The second cohomology of regular semisimple Hessenberg varieties from GKM theory}

\author[A. Ayzenberg]{Anton Ayzenberg}
\address{Faculty of computer science, National Research University Higher School of Economics, Moscow, Russian Federation}
\email{ayzenberga@gmail.com}

\author[M. Masuda]{Mikiya Masuda}
\address{Osaka City University Advanced Mathematical Institute, Sumiyoshi-ku, Osaka 558-8585, Japan.}
\email{masuda@sci.osaka-cu.ac.jp}

\author[T. Sato]{Takashi Sato}
\address{Research Institute for Mathematical Sciences, Kyoto University, Kyoto 606-8502, Japan, and Osaka City University Advanced Mathematical Institute, Sumiyoshi-ku, Osaka 558-8585, Japan.}
\email{00tkshst00@gmail.com}

\date{\today}



\begin{abstract}
We describe the second cohomology of a regular semisimple Hessenberg variety by generators and relations explicitly in terms of GKM theory.
The cohomology of a regular semisimple Hessenberg variety becomes a module of a symmetric group $\Sn$ by the dot action introduced by Tymoczko.
As an application of our explicit description, we give a formula describing the isomorphism class of the second cohomology as an $\Sn$-module.
Our formula is not exactly the same as the known formula by Chow or Cho-Hong-Lee but they are equivalent.
We also discuss its higher degree generalization.
\end{abstract}

\maketitle

\setcounter{tocdepth}{1}

\section{Introduction} \label{sect:1}
Let $\flag(n)$ denote the variety of all complete flags in $\C^n$.
A regular semisimple Hessenberg variety $\Hess(S,h)$ is a smooth subvariety of $\flag(n)$.
It is determined by a square matrix $S$ of size $n$ with distinct eigenvalues and a function $h$ (called a {\it Hessenberg function}) from the set of integers $[n]=\{1,\dots,n\}$ to itself satisfying the condition:
\[
	h(1)\le h(2)\le \cdots \le h(n)\qquad\text{and}\qquad h(j)\ge j \quad (\forall j\in [n]).
\]
The topology of $\Hess(S,h)$ depends only on $h$, i.e.~ does not depend on the choice of $S$.
The maximal $\C^*$-torus $T$ in the general linear group ${\rm GL}_n(\C)$, which commutes with $S$, naturally acts on $\Hess(S,h)$.
Using this $T$-action, one can study the cohomology $H^*(\Hess(S,h))$.
Specifically, Tymoczko (\cite{tymo08}) constructed an $\Sn$-action (called the {\it dot action}) on $H^*(\Hess(S,h))$ making it an $\Sn$-module.

A theorem of Brosnan-Chow \cite{br-ch} (the solution of Shareshian-Wachs conjecture \cite{sh-wa}) says that the graded $\Sn$-module $H^*(\Hess(S,h))$ is equivalent to the (graded) chromatic symmetric function $X_{G_h}(\mathbf{x},t)$ of a graph $G_h$ associated to $h$,
where $X_{G_h}(\mathbf{x}, t)$ is a polynomial in $t$ with symmetric functions in infinitely many variables $\mathbf{x}=(x_1,x_2,\cdots)$ as coefficients.
Moreover, it is shown in \cite{guay} that the Stanley-Stembridge conjecture on the $e$-positivity of $(3+1)$-incomparability graphs is reduced to showing the $e$-positivity of the graph $G_h$ for any Hessenberg function $h$.
Thus, we are led to the study of $H^*(\Hess(S,h))$.

In this paper, we investigate the second cohomology $H^2(\Hess(S,h))$ using GKM theory \cite{go-ko-ma} when $\Hess(S,h)$ is connected.
We exhibit its generators explicitly in terms of GKM theory and give a formula describing the isomorphism class of the $\Sn$-module in terms of $h$.
Prior to our work, Chow \cite{chow21} gave a formula for the coefficient of $t$ in $X_{G_h}(\mathbf{x},t)$ using $P$-tableaux.
Through the theorem by Brosnan-Chow mentioned above, Chow's formula is equivalent to ours.
After Chow's work, Cho-Hong-Lee \cite{CHL21} exhibited generators of $H^2(\Hess(S,h))$ geometrically using the Bialynicki-Birula decomposition of $\Hess(S,h)$ and gave a formula describing the isomorphism class of the $\Sn$-module $H^2(\Hess(S,h))$.
Their formula is also equivalent to ours.
However, the methods are different and the relation between their generators of $H^2(\Hess(S,h))$ and ours is unclear.

Since we obtain explicit generators of $H^2(\Hess(S,h))$ in this paper, we are able to characterize Hessenberg functions $h$ for which the whole cohomology ring $H^*(\Hess(S,h))$ is generated in degree two as a graded ring.
It turns out that the graph $G_h$ associated to such $h$ is what is called a (double) lollipop (\cite{da-wi}, \cite{hu-na-yo}).
We will discuss this subject in a forthcoming paper \cite{AMS2}.

As mentioned above, we describe $H^2(\Hess(S,h))$ in terms of explicit generators and relations when $\Hess(S,h)$ is connected.
It is known that $\Hess(S,h)$ is connected, in other words, the  restriction map $\iota^*\colon H^0(\flag(n))\to H^0(\Hess(S,h))$ is an isomorphism if and only if $h(j)\ge j+1$ for any $j\in [n-1]$.
As a generalization of this setting, we consider the case where $h(j)\ge j+d$ for any $j\in [n-d]$, where $d$ is an integer $\ge 2$.
In this case, we show that the retriction map $\iota^*\colon H^{2p}(\flag(n))\to H^{2p}(\Hess(S,h))$ is an isomorphism for $p<d$ and describe $H^{2d}(\Hess(S,h))$ in terms of explicit generators and relations.

The organization of the paper is as follows.
In Section~\ref{sect:2}, we review necessary facts on regular semisimple Hessenberg varieties and labeled graphs associated to them.
Section~\ref{sect:3} discusses an inductive formula to compute the Poincar\'e polynomial of $\Hess(S,h)$ and a formula on the second Betti number of $\Hess(S,h)$ in terms of $h$.
In Section~\ref{sect:4}, we provide three types of elements in the $T$-equivariant cohomology $H_T^*(\Hess(S,h))$ using GKM theory, discuss relations among them, and observe the dot action of $\Sn$ on them.
We prove that these three types of elements together with $H^2(BT)$ generate $H^2_T(\Hess(S,h))$.
In Section~\ref{sect:5}, we explicitly describe $H^2(\Hess(S,h))$ in terms of generators and relations, and give a formula describing the isomorphism class of the $\Sn$-module $H^2(\Hess(S,h))$ in terms of $h$.
In Section~\ref{sect:6} we discuss the generalization of the result on the second cohomology to the higher degree cohomology mentioned above.

\section{Regular semisimple Hessenberg varieties} \label{sect:2}

\subsection{Hessenberg variety}
The flag variety $\flag(n)$ is defined as the set of nested linear subspaces of $\C^n$:
\[
	\flag(n) = \{ V_\bullet=(V_1 \subset V_2 \subset \cdots \subset V_n = \C^n) \mid \dim_{\C} V_i = i \quad \forall i\in [n]\}.
\]
Given a square matrix $A$ of size $n$ and a function $h\colon [n]\to [n]$ (called a \textit{Hessenberg function}) satisfying
\[
	h(1)\le h(2)\le \cdots\le h(n)\qquad\text{and}\qquad h(j)\ge j\quad (\forall j\in [n]),
\]
the Hessenberg variety $\Hess(A,h)$ is defined by
\[
	\Hess(A,h):=\{V_\bullet\in \flag(n)\mid A(V_j)\subset V_{h(j)}\quad\text{for $\forall j\in [n]$}\},
\]
where the matrix $A$ is regarded as a linear transformation on $\C^n$.
We often express the Hessenberg function $h$ as a vector $(h(1),\dots,h(n))$ by listing the values of $h$.
When $h=(n,\dots,n)$, it is obvious from the definition that $\Hess(A,h)$ is the flag variety $\flag(n)$ regardless of the choice of $A$.

As illustrated in the following example, we can visualize a Hessenberg function $h$ by drawing a configuration of the shaded boxes on a square grid of size $n\times n$, which consists of boxes in the $i$-th row and the $j$-th column satisfying $i\le h(j)$.
Since it is assumed that $j\le h(j)$ for any $j\in [n]$, the essential part is the shaded boxes below the diagonal.

\begin{exam}\label{example:HessenbergFunction}
Let $n=5$.
The Hessenberg function $h=(3,3,4,5,5)$ corresponds to the configuration of the shaded boxes drawn on the left grid and the essential shaded boxes (i.e.~ below the diagonal) are drawn on the right grid in Figure $\ref{pic:stair-shape}$.
\end{exam}

\begin{figure}[H]
\begin{center}
\setlength{\unitlength}{5mm}
\begin{picture}(14,6)(0,-1)
	\multiput(0,4.2)(0,-1){3}{\colorbox{gray7}{\phantom{\vrule width 3mm height 3mm}}}
	\multiput(1,4.2)(0,-1){3}{\colorbox{gray7}{\phantom{\vrule width 3mm height 3mm}}}
	\multiput(2,4.2)(0,-1){4}{\colorbox{gray7}{\phantom{\vrule width 3mm height 3mm}}}
	\multiput(3,4.2)(0,-1){5}{\colorbox{gray7}{\phantom{\vrule width 3mm height 3mm}}}
	\multiput(4,4.2)(0,-1){5}{\colorbox{gray7}{\phantom{\vrule width 3mm height 3mm}}}
	\linethickness{0.3mm}
	\multiput(1,0)(1,0){4}{\line(0,1){5}}
	\multiput(0,1)(0,1){4}{\line(1,0){5}}
	\put(0,0){\framebox(5,5)}
	\multiput(9,3.2)(0,-1){2}{\colorbox{gray7}{\phantom{\vrule width 3mm height 3mm}}}
	\multiput(10,2.2)(0,-1){1}{\colorbox{gray7}{\phantom{\vrule width 3mm height 3mm}}}
	\multiput(11,1.2)(0,-1){1}{\colorbox{gray7}{\phantom{\vrule width 3mm height 3mm}}}
	\multiput(12,0.2)(0,-1){1}{\colorbox{gray7}{\phantom{\vrule width 3mm height 3mm}}}
	\linethickness{0.3mm}
	\multiput(10,0)(1,0){4}{\line(0,1){5}}
	\multiput(9,1)(0,1){4}{\line(1,0){5}}
	\put(9,0){\framebox(5,5)}
\end{picture}
\end{center}
\caption{The configuration corresponding to $h=(3,3,4,5,5)$.}
\label{pic:stair-shape}
\end{figure}
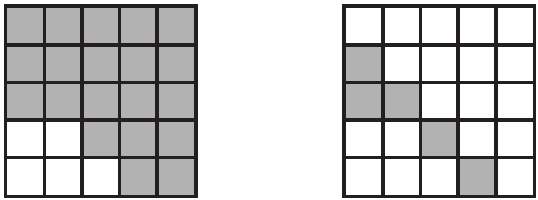


The Hessenberg variety $\Hess(S,h)$ for a square matrix $S$ of size $n$ with distinct eigenvalues is called \textit{regular semisimple}.
\begin{theo}[\cite{ma-pr-sh}] \label{theo:MPS}\hfill
\begin{enumerate}
\item $\Hess(S,h)$ is smooth.
\item $\dim_\C\Hess(S,h)=\sum_{j=1}^n(h(j)-j)$.
\item $\Hess(S,h)$ is connected if and only if $h(j)\ge j+1$ for $\forall j\in [n-1]$.
\item $H^{odd}(\Hess(S,h))=0$ and the $2k$-th Betti number of $\Hess(S,h)$ is given by
\[
	\#\{w\in \mathfrak{S}_n\mid \ell_h(w)=k\}
\]
where
\begin{equation}\label{eqDhw}
	\ell_h(w)=\#\{ 1\le j<i\le n\mid w(j)>w(i),\ i\le h(j)\}.
\end{equation}
\end{enumerate}
\end{theo}

Since $S$ commutes with a maximal torus $T$ of ${\rm GL}_n(\C)$, the restricted action of $T$ on $\flag(n)$ leaves $\Hess(S,h)$ invariant.
One sees that
\begin{equation} \label{eq:Hess(S,h)_fixed}
	\Hess(S,h)^T=\flag(n)^T=\Sn.
\end{equation}

\subsection{Equivariant cohomology}
We shall briefly review equivariant cohomology.
For a $T$-space $X$, the equivariant cohomology $H^*_T(X)$ is defined as
\[
	H^*_T(X):=H^*(ET\times_T X)
\]
where $ET\to BT$ is the universal principal $T$-bundle and $ET\times_T X$ is the orbit space of $ET\times X$ by the diagonal $T$-action.
Since $T$ is isomorphic to $(\C^*)^n$, $BT$ is homeomorphic to $(\C P^\infty)^n$ and hence $H^*(BT)$ is a polynomial ring in $n$ elements of $H^2(BT)$ which form a generator of $H^2(BT)$.
Note that for a one-point space $pt$, we have
\[
	H^*_T(pt)=H^*(BT).
\]
Since the $T$-action on $ET$ is free, the projection $ET\times X\to ET$ on the first factor induces a fibration
\[
	X\xrightarrow{\iota} ET\times_T X\xrightarrow{\pi} ET/T=BT.
\]
The equivariant cohomology $H^*_T(X)$ is not only a ring but also an algebra over $H^*(BT)$ through $\pi^*\colon H^*(BT)\to H^*_T(X)$.
As easily seen, the restriction map $\iota^*\colon H^*_T(X)\to H^*(X)$ sends $H^2(BT)$ to zero.
Therefore, it induces a ring homomorphism
\begin{equation} \label{eq:iota*}
	H^*_T(X)/(H^2(BT))\to H^*(X),
\end{equation}
where $(H^2(BT))$ denotes the ideal generated by $\pi^*(H^2(BT))$.
If $H^{odd}(X)=0$ (this is the case when $X=\Hess(S,h)$), then the map \eqref{eq:iota*} above is an isomorphism.

\subsection{GKM theory and labeled graph}
We choose and fix a set of generators of $H^2(BT)$ and denote them by $t_1,\dots,t_n$, they correspond to the choice of coordinates in the ambient space $\C^n$ of flags.
Since $H^{odd}(\Hess(S,h))=0$, the torus action is cohomologically equivariantly formal~\cite{go-ko-ma}.
Therefore the restriction map to the $T$-fixed point set
\begin{align*}
	H^*_T(\Hess(S,h))\to H^*_T(\Hess(S,h)^T)
	&=\bigoplus_{w\in \Sn}H^*_T(w) =\bigoplus_{w\in \Sn}\Z[t_1,\dots,t_n]\\
	&={\rm Map}(\Sn,\Z[t_1,\dots,t_n])
\end{align*}
is injective, where ${\rm Map}(P,Q)$ denotes the set of all maps from $P$ to $Q$.
Since the restriction map above is injective, we think of $H^*_T(\Hess(S,h))$ as a subset of ${\rm Map}(\Sn,\Z[t_1,\dots,t_n])$.

\begin{prop}[\cite{tymo08}] \label{prop:GKM}
An element $f\in {\rm Map}(\Sn,\Z[t_1,\dots,t_n])$ is in $H^*_T(\Hess(S,h))$ if and only if
\[
	f(v)\equiv f(w) \pmod{t_{w(i)}-t_{w(j)}}\quad \text{whenever $v=w\cdot(i,j)$ for $j<i\le h(j)$}
\]
where $(i,j)$ denotes the transposition interchanging $i$ and $j$.
\end{prop}

To a Hessenberg function $h$, one associates a graph $(V,E)$ with a label on the edge set $E$
\[
	\alpha\colon E\to H^2(BT)\backslash\{0\}
\]
where
\begin{enumerate}
	\item $V=\Sn$,
	\item $E=\{ \{v,w\} \mid v,w\in V,\ v=w\cdot(i,j)\text{ for some } j<i\le h(j)\}$,
	\item $\alpha(\{v,w\})=t_{w(i)}-t_{w(j)}$ up to sign for $v = w\cdot(i,j)$.
\end{enumerate}
We denote the triple $(V,E,\alpha)$ by $\Gamma(h)$ and call a \textit{labeled graph} associated to $h$.
This is a slight variant of the notion of GKM graph.
The set of elements in ${\rm Map}(\Sn,\Z[t_1,\dots,t_n])$ satisfying the congruence relation in Proposition~\ref{prop:GKM} is sometimes called the \textit{graph cohomology} of $\Gamma(h)$.
Proposition~\ref{prop:GKM} says that $H^*_T(\Hess(S,h))$ agrees with the graph cohomology of $\Gamma(h)$.
In particular, $H^*_T(\Hess(S,h))$ is independent of the choice of the regular semisimple matrix $S$.
Notice, that even the equivariant diffeomorphism type of $\Hess(S,h)$ is independent of $S$, see~\cite{Abe_Horiguchi:2019}.

We often think of $t_i$ as an element of ${\rm Map}(\Sn,\Z[t_1,\dots,t_n])$ by regarding it as a constant map.
Obviously, $t_i$ satisfies the congruence relation in Proposition~\ref{prop:GKM}, so it is in $H^2_T(\Hess(S,h))$.
Then
\[
	H^*(\Hess(S,h))=H^*_T(\Hess(S,h))/(t_1,\dots,t_n)
\]
because $H^{odd}(\Hess(S,h))=0$, where $(t_1,\dots,t_n)$ is the ideal generated by $t_1,\dots,t_n$.

\begin{exam} \label{exam:GKM_graph}
Let $n=3$.
For $h=(2,3,3)$ and $h'=(3,3,3)$, the corresponding labeled graphs $\Gamma(h)$ and $\Gamma(h')$ are depicted in Figure $\ref{pic:GKM graphs}$ where we use the one-line notation for each vertex.

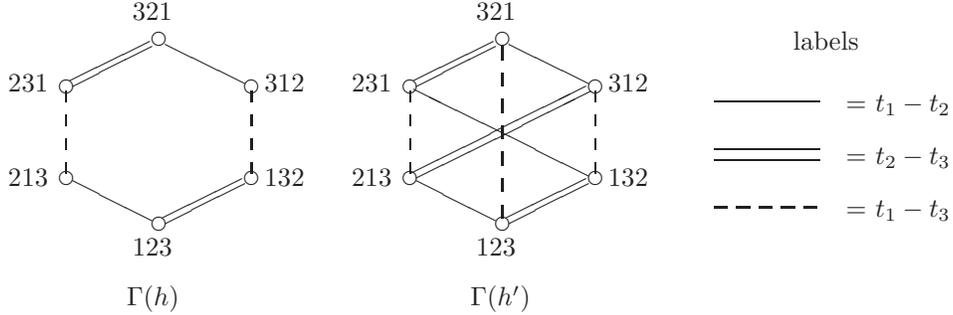
\begin{figure}[H]
\begin{center}
\begin{picture}(400,90)
\put(60,14){\circle{5}}
\put(60,84){\circle{5}}
\put(95,31){\circle{5}}
\put(95,66){\circle{5}}
\put(25,31){\circle{5}}
\put(25,66){\circle{5}}

\put(57.5,15){\line(-2,1){30}}
\put(92.5,67){\line(-2,1){30}}
\put(61.5,16){\line(2,1){30}}
\put(62.5,14){\line(2,1){30}}
\put(26.5,68){\line(2,1){30}}
\put(27.5,66){\line(2,1){30}}

\multiput(25,33)(0,9){4}{\line(0,1){4}}
\multiput(95,33)(0,9){4}{\line(0,1){4}}

\put(50,2){$123$}
\put(50,91){$321$}
\put(100,28){$132$}
\put(100,64){$312$}
\put(3,28){$213$}
\put(3,64){$231$}

\put(48,-17){$\Gamma(h)$}

\put(190,14){\circle{5}}
\put(190,84){\circle{5}}
\put(225,31){\circle{5}}
\put(225,66){\circle{5}}
\put(155,31){\circle{5}}
\put(155,66){\circle{5}}

\put(222.5,32){\line(-2,1){66}}
\put(187.5,15){\line(-2,1){30}}
\put(222.5,67){\line(-2,1){30}}
\put(155.5,33){\line(2,1){66}}
\put(157.5,31){\line(2,1){66}}
\put(191.5,16){\line(2,1){30}}
\put(192.5,14){\line(2,1){30}}
\put(156.5,68){\line(2,1){30}}
\put(157.5,66){\line(2,1){30}}
\multiput(190,16)(0,10){7}{\line(0,1){5}}
\multiput(155,33)(0,9){4}{\line(0,1){4}}
\multiput(225,33)(0,9){4}{\line(0,1){4}}

\put(180,2){$123$}
\put(180,91){$321$}
\put(230,28){$132$}
\put(230,64){$312$}
\put(133,28){$213$}
\put(133,64){$231$}

\put(178,-17){$\Gamma(h')$}

\put(300,80){{\rm labels}}
\put(270,60){\line(1,0){40}}
\put(320,56.5){= $t_1-t_2$}
\put(270,42){\line(1,0){40}}
\put(270,38){\line(1,0){40}}
\put(320,36.5){= $t_2-t_3$}
\multiput(270,20)(8.5,0){5}{\line(1,0){5}}
\put(320,16.5){= $t_1-t_3$}
\end{picture}
\end{center}
\vspace{15pt}
\caption{The labeled graphs $\Gamma(h)$ and $\Gamma(h')$}
\label{pic:GKM graphs}
\end{figure}
\end{exam}

Both graphs $\Gamma(h)$ and $\Gamma(h')$ in Example~\ref{exam:GKM_graph} are connected.
In general, it is not difficult to see that $\Gamma(h)$ is connected if and only if $h(j)\ge j+1$ for any $j\in [n-1]$.

\section{The second Betti number of \texorpdfstring{$\Hess(S,h)$}{Hess(S,h)}}
\label{sect:3}

In this section, we give an inductive formula to compute the Poincar\'e polynomial of $\Hess(S,h)$ using Theorem~\ref{theo:MPS}(4)
and apply it to obtain an explicit formula of the second Betti number of $\Hess(S,h)$ in terms of the Hessenberg function $h$.

\subsection{An inductive formula}
For a space $X$ such that $H^{odd}(X)=0$ and the rank of $H^*(X)$ over $\Z$ is finite, we define
\[
	\Poin(X,\sqrt{q}):=\sum_{r=0}^\infty b_{2r}(X)q^r
\]
where $b_{2r}(X)$ denotes the $2r$-th Betti number over $\Z$.
Let $h^j$ be the Hessenberg function obtained by removing all the boxes in the $j$-th row and all the boxes in the $j$-th column (See Figure \ref{pic:h^j}).
To be precise,
\[
h^j(i) =
\begin{cases}
h(i) & (i < j,\ h(i)<j)\\
h(i)-1 & (i < j,\ h(i) \geq j)\\
h(i+1) -1 & (i\geq j)
\end{cases}
\]
\begin{figure}[H]
\begin{center}
\setlength{\unitlength}{5mm}
\begin{picture}(24,8)(-3,-1)
	\linethickness{0.3mm}
	\put(-4,3.3){$j$-th row $\rightarrow$}
	\put(1.35,5.5){$\downarrow$}
	\put(0.5,6.3){$j$-th column}
	\put(2.3,-1.2){$h$}
	\multiput(0,4.2)(0,-1){3}{\colorbox{gray7}{\phantom{\vrule width 3mm height 3mm}}}
	\multiput(1,4.2)(0,-1){3}{\colorbox{gray7}{\phantom{\vrule width 3mm height 3mm}}}
	\multiput(2,4.2)(0,-1){4}{\colorbox{gray7}{\phantom{\vrule width 3mm height 3mm}}}
	\multiput(3,4.2)(0,-1){5}{\colorbox{gray7}{\phantom{\vrule width 3mm height 3mm}}}
	\multiput(4,4.2)(0,-1){5}{\colorbox{gray7}{\phantom{\vrule width 3mm height 3mm}}}
	\multiput(1,0)(1,0){4}{\line(0,1){5}}
	\multiput(0,1)(0,1){4}{\line(1,0){5}}
	\put(0,0){\framebox(5,5){}}
	\put(6.75,2.3){$\leadsto$}
	\put(5.9,1.7){remove}
	\put(10.15,4.3){$\leftarrow$}
	\put(10.15,3.3){$\nwarrow$}
	\put(9.3,3.3){$\uparrow$}
	\put(9,4.2){\colorbox{gray7}{\phantom{\vrule width 3mm height 3mm}}}
	\put(11,4.2){\colorbox{gray7}{\phantom{\vrule width 3mm height 3mm}}}
	\put(12,4.2){\colorbox{gray7}{\phantom{\vrule width 3mm height 3mm}}}
	\put(13,4.2){\colorbox{gray7}{\phantom{\vrule width 3mm height 3mm}}}
	\multiput(9,2.2)(0,-2){1}{\colorbox{gray7}{\phantom{\vrule width 3mm height 3mm}}}
	\multiput(11,2.2)(0,-1){2}{\colorbox{gray7}{\phantom{\vrule width 3mm height 3mm}}}
	\multiput(12,2.2)(0,-1){3}{\colorbox{gray7}{\phantom{\vrule width 3mm height 3mm}}}
	\multiput(13,2.2)(0,-1){3}{\colorbox{gray7}{\phantom{\vrule width 3mm height 3mm}}}
	\multiput(12,0)(1,0){2}{\line(0,1){3}}
	\multiput(12,4)(1,0){2}{\line(0,1){1}}
	\multiput(9,1)(0,1){2}{\line(1,0){1}}
	\multiput(11,1)(0,1){2}{\line(1,0){3}}
	\put(9,4){\framebox(1,1){}}
	\put(9,0){\framebox(1,3){}}
	\put(11,4){\framebox(3,1){}}
	\put(11,0){\framebox(3,3){}}
	\put(15.75,2.3){$\leadsto$}
	\put(19.8,-1.2){$h^j$}
	\multiput(18,4.2)(0,-1){2}{\colorbox{gray7}{\phantom{\vrule width 3mm height 3mm}}}
	\multiput(19,4.2)(0,-1){3}{\colorbox{gray7}{\phantom{\vrule width 3mm height 3mm}}}
	\multiput(20,4.2)(0,-1){4}{\colorbox{gray7}{\phantom{\vrule width 3mm height 3mm}}}
	\multiput(21,4.2)(0,-1){4}{\colorbox{gray7}{\phantom{\vrule width 3mm height 3mm}}}
	\multiput(19,1)(1,0){3}{\line(0,1){4}}
	\multiput(18,2)(0,1){3}{\line(1,0){4}}
	\put(18,1){\framebox(4,4)}
\end{picture}
\end{center}
\caption{The configuration corresponding to $h^j$.}
\label{pic:h^j}
\end{figure}
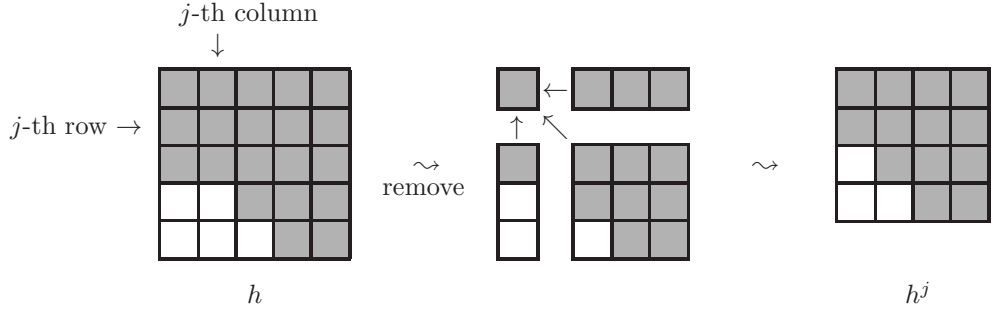

\begin{prop} \label{prop:inductive_formula}
Under the above understanding, we have
\[
	\Poin(\Hess(S,h),\sqrt{q})=\sum_{j=1}^nq^{h(j)-j}\Poin(\Hess(S',h^j),\sqrt{q})
\]
where $S'$ denotes a matrix of size $n-1$ with distinct eigenvalues.
\end{prop}

\begin{proof}
It follows from Theorem~\ref{theo:MPS}(4) that
\begin{equation} \label{eq:Poincare}
	\Poin(\Hess(S,h),\sqrt{q})=\sum_{w\in \Sn}q^{\ell_h(w)}
\end{equation}
where
\[
	\ell_h(w)=\#\{1\le j<i\le n\mid w(j)>w(i),\ i\le h(j)\}.
\]
For $j\in [n]$ we set
\[
	\Sn^j:=\{w\in \Sn\mid w(j)=n\}
\]
and consider the following decomposition of $\Sn$:
\[
	\Sn=\Sn^1\sqcup \Sn^2\sqcup \cdots \sqcup \Sn^n.
\]
If $w$ is in $\Sn^j$, then $w(j)=n$ so that $w(j)>w(i)$ for any $j<i\le h(j)$.
Therefore, if we denote by $w^j$ the permutation on $[n-1]$ obtained by removing $w(j)=n$ from $w$, then we have
\[
	\ell_h(w)=h(j)-j+\ell_{h^j}(w^j).
\]
This together with \eqref{eq:Poincare} implies the formula in the proposition.
\end{proof}

For a Hessenberg function $h$, we consider the following two sets:
\begin{equation} \label{eq:JL}
\begin{split}
	\J(h):&=\{ j\in [n-1]\mid h(j-1)=h(j)=j+1\}\\
	\L(h):&=\{ j\in [n-1] \mid h(j-1)=j\text{ and }h(j)=j+1\}
\end{split}
\end{equation}
where we understand $h(0)=1$.

\begin{lemm} \label{lemm:2_betti}
Suppose that $h(j)\ge j+1$ for $j\in[n-1]$.
Then the second Betti number $b_2(\Hess(S,h))$ of $\Hess(S,h)$ can be expressed as
\begin{equation*}
	b_2(\Hess(S,h))=\sum_{j\in \L(h)}\binom{n}{j}+(n-1)|\J(h)|-|\L(h)|.
\end{equation*}
\end{lemm}

\begin{exam}
For the Hessenberg function $h=(3,3,4,5,5)$ in Example~\ref{example:HessenbergFunction}, we have
\[
	\J(h)=\{2\}\quad\text{and}\quad \L(h)=\{3,4\}.
\]
Therefore
\[
	b_2(\Hess(S,h))=\binom{5}{3}+\binom{5}{4}+(5-1)\cdot 1-2=17.
\]
\end{exam}

\begin{proof}[Proof of Lemma~\ref{lemm:2_betti}]
We prove the lemma by induction on $n$.
When $n=2$, $h=(2,2)$ since $h(j)\ge j+1$ for $j\in [n-1]$ by assumption.
Therefore, $\Hess(S,h)$ is $\C P^1$ while $\J(h)=\emptyset$ and $\L(h)=\{1\}$.
This shows that the lemma holds when $n=2$.

Suppose that $n\ge 3$.
It follows from Proposition~\ref{prop:inductive_formula} that
\begin{equation} \label{eq:b2h}
	b_2(\Hess(S,h))=\sum_{h(j)=j+1}b_0(\Hess(S',h^j))+b_2(\Hess(S',h^n)).
\end{equation}
Here, $\Hess(S',h^j)$ is connected if $h(j-1)= j+1$ and has $\binom{n-1}{j-1}$ many connected components if $h(j-1)=j$ (the detailed description of connected components can be found in~\cite{teff}).
This means that
\begin{equation} \label{eq:b0}
	b_0(\Hess(S',h^j))=
	\begin{cases}
		1\quad&(\text{if }j\in \J(h))\\
		\binom{n-1}{j-1}\quad& (\text{if } j\in \L(h))\\
	\end{cases}
\end{equation}
since $h(j)=j+1$ for $j$ in the sum of \eqref{eq:b2h}.

On the other hand, by the induction assumption, we have
\begin{equation} \label{eq:b2hn}
	b_2(\Hess(S',h^n))=\sum_{j\in \L(h^n)}\binom{n-1}{j}+(n-2)|\J(h^n)|-|\L(h^n)|.
\end{equation}
Here, by looking at the $(n-2,n-1)$ and $(n-1,n)$ boxes in the configuration associated to $h$, one sees that
\begin{equation} \label{eq:Jhn}
	\J(h^n)=
	\begin{cases}
		\J(h)\ &(\text{if }h(n-2)=n-1)\\
		\J(h)\backslash\{n-1\}\ &(\text{if }h(n-2)=n,\ h(n-3)=n-2)\\
		\left(\J(h)\backslash\{n-1\}\right)\cup\{n-2\}\ &(\text{if } h(n-2)=n,\ h(n-3)\ge n-1)
	\end{cases}
\end{equation}
\begin{equation} \label{eq:Lhn}
	\L(h^n)=
	\begin{cases}
		\L(h)\backslash\{n-1\}\quad&(\text{if }h(n-2)=n-1)\\
		\L(h)\cup\{n-2\}\quad&(\text{if }h(n-2)=n,\ h(n-3)=n-2)\\
		\L(h)\quad&(\text{if }h(n-2)=n,\ h(n-3)\ge n-1).
	\end{cases}
\end{equation}

We consider three cases according to the cases above and plug \eqref{eq:Jhn} and \eqref{eq:Lhn} in the right hand side of \eqref{eq:b2hn} in each case.
Then, this together with \eqref{eq:b0} will show that the right hand side of \eqref{eq:b2h} agrees with the right hand side of the identity in the lemma.
For instance, when $h(n-2)=n-1$, it follows from \eqref{eq:b0}, \eqref{eq:b2hn}, \eqref{eq:Jhn}, and \eqref{eq:Lhn} that the right hand side of \eqref{eq:b2h} turns into
\begin{align*}
	&\sum_{j\in \L(h)}\binom{n-1}{j-1}+|\J(h)|\\
	&+\sum_{j\in \L(h)\backslash\{n-1\}}\binom{n-1}{j}+(n-2)|\J(h)|-|\L(h)\backslash\{n-1\}|\\
	=& \sum_{j\in \L(h)}\left(\binom{n-1}{j-1}+\binom{n-1}{j}\right)-\binom{n-1}{n-1}+(n-1)|\J(h)|-|\L(h)|+1\\
	=&\sum_{j\in \L(h)}\binom{n}{j}+(n-1)|\J(h)|-|\L(h)|,
\end{align*}
which coincides with the right hand side of the identity in the lemma.
The other two cases can be proved similarly, so we omit the proof.
\end{proof}

\section{Generators of \texorpdfstring{$H^2(\Hess(S,h))$}{H2(Hess(S,h))}} \label{sect:4}

\subsection{Three types of elements in $H^2_T(\Hess(S,h))$}
The following three types of elements play a role in our argument.

\begin{lemm} \label{lemm:xytau}
The elements $x_i, y_{j,k}, \tau_A$ in ${\rm Map}(\Sn,\Z[t_1,\dots,t_n])$ defined below are in $H^2_T(\Hess(S,h))$.
\begin{enumerate}
	\item For $i\in [n]$, $x_i(w):=t_{w(i)}$.
	\item For $j\in \J(h)$ and $k\in [n]$,
	\[
		y_{j,k}(w):=
		\begin{cases}
			t_k-t_{w(j+1)}\quad &(\text{if } k\in \{w(1),\dots,w(j)\})\\
			0\quad &(\text{otherwise}).
		\end{cases}
	\]
	\item For $A\subset [n]$ with $|A| \in \L(h)$,
	\begin{equation*} \label{eq:tauA}
		\tau_A(w):=
		\begin{cases}
			t_{w(j)}-t_{w(j+1)}\quad &(\text{if }\{w(1),\dots,w(j)\}=A)\\
			0 \quad &(\text{otherwise}).
		\end{cases}
	\end{equation*}
\end{enumerate}
\end{lemm}

\begin{proof}
What we have to do is to check that the elements in the lemma satisfy the congruence relation in Proposition~\ref{prop:GKM}.
The check for $x_i$ is straightforward and that for $y_{j,k}$ is done in \cite[Lemma 10.2]{AHHM} in a more general setting.
Therefore, we shall check the congruence relation for $\tau_A$.
We set $j = |A|$.

Suppose that $v=w\cdot(r,s)$ for $r<s\le h(r)$.
We consider three cases.

Case 1.
The case where $r=j$.
In this case, $s=j+1$ because $h(r)=h(j)=j+1$, so that $t_{w(r)}-t_{w(s)}=t_{w(j)}-t_{w(j+1)}$.
Therefore, looking at the definition of $\tau_A$, we see that the congruence relation is obviously satisfied in this case.

Case 2.
The case where $r\le j-1$.
In this case, $s\le j$ since $s\le h(r)\le h(j-1)=j$ where $h(j-1)=j$ is because $j\in \L(h)$.
Therefore, $\{w(1),\dots,w(j)\}=\{v(1),\dots,v(j)\}$ and hence $\tau_A(w)=\tau_A(v)$, thereby the congruence relation is trivially satisfied.

Case 3.
The case where $r\ge j+1$.
In this case, $s\ge j+2$ as $s>r$.
Therefore, $\{w(1),\dots,w(j)\}=\{v(1),\dots,v(j)\}$ in this case, too.
Hence the congruence relation is satisfied.
\end{proof}

\begin{rema} \label{rema:4-1}
(1) The elements $x_i$ lie in $H^2_T(\Hess(S,h))$ for any Hessenberg function.
Indeed, when $h=(n,\dots,n)$, $\Hess(S,h)$ is the flag variety $\flag(n)$ and $x_i$ is the equivariant first Chern class of the line bundle $E_i/E_{i-1}$ where $E_i$ is the $i$-th tautological vector bundle over the flag variety $\flag(n)$:
\[
	E_i:=\{(V_\bullet,v) \in \flag(n)\times \C^n \mid v\in V_i\}.
\]

(2) The element $y_{j,k}$ exists in $H^2_T(\Hess(S,h))$ for any $j$ with $h(j)=j+1$ by the same definition as above but for our purpose it suffices to consider $y_{j,k}$ for $j\in \J(h)$.
\end{rema}

\begin{lemm} \label{lemm:relation}
For $x_i, y_{j,k}, \tau_A$ in Lemma~\ref{lemm:xytau}, the following holds.
\begin{enumerate}
\item $\sum_{i=1}^n x_i=\sum_{i=1}^n t_i$.
\item $\sum_{k=1}^ny_{j,k}=\sum_{i=1}^j i(x_i-x_{i+1})=x_1+\cdots+x_j-jx_{j+1}$.
\item $\sum_{|A|=j}\tau_A=x_j-x_{j+1}$.
\item Let $m$ be the maximum element in $\J(h)$ if $\J(h)\not=\emptyset$ and $0$ otherwise.
Then
\[
	y_{m,k}+\sum_{k\in A,\, m<|A|\le n-1}\tau_A=t_k-x_n\qquad \text{for any $k\in [n]$}
\]
where we understand $y_{0,k}=0$ and the sum is $0$ when $m=n-1$.
\end{enumerate}
\end{lemm}

\begin{proof}
For each identity, we check that the left hand side and the right hand side take the same value at every $w\in \Sn$.
The check for the identities in (1)--(3) is straightforward, so we leave it for the reader and shall check the identity in (4).

It follows from the definition of $\tau_A$ in Lemma~\ref{lemm:xytau} that
\[
	\left(\sum_{k\in A,\, |A|=j}\tau_A\right)(w)=
	\begin{cases}
		t_{w(j)}-t_{w(j+1)}\quad&(\text{if }k\in \{w(1),\dots,w(j)\})\\
		0\quad&(\text{otherwise}).
	\end{cases}
\]
Therefore,
\[
	\left(\sum_{k\in A,\, m<|A|\le n-1}\tau_A\right)(w)=
	\begin{cases}
		t_{w(m+1)}-t_{w(n)}\quad&(\text{if }k\in \{w(1),\dots,w(m)\})\\
		t_{k}-t_{w(n)}\quad&(\text{if }k\in \{w(m+1),\dots,w(n-1)\})\\
		0\quad&(\text{otherwise}).
	\end{cases}
\]
This together with the definition of $y_{m,k}$, that is
\[
	y_{m,k}(w)=
	\begin{cases}
		t_k-t_{w(m+1)}\quad&(\text{if }k\in \{w(1),\dots,w(m)\})\\
		0\quad&(\text{otherwise}),
	\end{cases}
\]
shows that the left hand side at the identity in (4) evaluated at $w$ agrees with $t_k-t_{w(n)}$, proving the desired identity.
\end{proof}

\subsection{Dot action}
We consider an action of $\sigma\in \Sn$ on $\Z[t_1,\dots,t_n]$ sending $t_i$ to $t_{\sigma(i)}$ for $i\in [n]$ and define an action of $\sigma\in \Sn$ on $f\in {\rm Map}(\Sn,\Z[t_1,\dots,t_n])$ by
\begin{equation} \label{eq:dot_action}
	(\sigma\cdot f)(w):=\sigma(f(\sigma^{-1}w)).
\end{equation}
One can check that if $f$ is in $H^*_T(\Hess(S,h))$, then so is $\sigma\cdot f$.
The action of $\Sn$ on $H^*_T(\Hess(S,h))$ preserves the ideal $(t_1,\dots,t_n)$ generated by $t_1,\dots,t_n$, so the action descends to an action of $\Sn$ on
\[
	H^*(\Hess(S,h))=H^*_T(\Hess(S,h))/(t_1,\dots,t_n).
\]
This action, called the \emph{dot action}, was introduced by Tymoczko \cite{tymo08}.

\begin{lemm} \label{lemm:action_on_xytau}
Let $x_i, y_{j,k}, \tau_A$ be as in Lemma~\ref{lemm:xytau}.
Then, for $\sigma\in \Sn$, we have
\[
	\sigma\cdot x_i = x_i,\quad \sigma\cdot y_{j,k}=y_{j,\sigma(k)},\quad
	\sigma\cdot\tau_A = \tau_{\sigma(A)}.
\]
\end{lemm}

\begin{proof}
The proof is straightforward.
Indeed, we have
\[
	(\sigma\cdot x_i)(w) = \sigma(x_i(\sigma^{-1}w))
	= \sigma(t_{\sigma^{-1}w(i)}) = t_{\sigma\sigma^{-1}w(i)} = t_{w(i)},
\]
proving $\sigma\cdot x_i=x_i$.
As for $y_{j,k}$, we have
\begin{align*}
	(\sigma\cdot y_{j,k})(w)&=\sigma(y_{j,k}(\sigma^{-1}w))\\
	&=\begin{cases}
		\sigma(t_k-t_{\sigma^{-1}w(j+1)})\quad 
		&(\text{if } k\in \{\sigma^{-1}w(1),\dots,\sigma^{-1}w(j)\})\\
		0\quad &(\text{otherwise})
	\end{cases}\\
	&=\begin{cases}
		t_{\sigma(k)}-t_{w(j+1)}\quad &(\text{if } \sigma(k)\in \{w(1),\dots,w(j)\})\\
		0\quad &(\text{otherwise}),
	\end{cases}\\
\end{align*}
proving $\sigma\cdot y_{j,k}=y_{j,\sigma(k)}$.
Similarly, as for $\tau_A$ with $|A|=j$, we have
\begin{align*}
	(\sigma\cdot\tau_A)(w) &= \sigma(\tau_A(\sigma^{-1}w))\\
	&= \begin{cases}
		\sigma(t_{\sigma^{-1}w(j)}-t_{\sigma^{-1}w(j+1)})\quad
		&(\text{if }\{\sigma^{-1}w(1),\dots,\sigma^{-1}w(j)\}=A)\\
		0 \quad &(\text{otherwise})
	\end{cases}\\
	&=\begin{cases}
		t_{w(j)}-t_{w(j+1)}\quad &(\text{if }\{w(1),\dots,w(j)\}=\sigma(A))\\
		0 \quad &(\text{otherwise}),
	\end{cases}
\end{align*}
proving $\sigma\cdot\tau_A=\tau_{\sigma(A)}$.
\end{proof}

\subsection{Generators of $H^2_T(\Hess(S,h))$}
Remember that
\begin{align*}
	\J(h)&=\{ j\in [n-1]\mid h(j-1)=h(j)=j+1\}\\
	\L(h)&=\{ j\in [n-1] \mid h(j-1)=j,\ h(j)=j+1\}
\end{align*}
where $h(0)=1$, and $y_{j,k}$ is defined for $j\in \J(h)$ and $\tau_A$ is defined for $A\subset [n]$ with $|A|\in \L(h)$.

\begin{prop} \label{prop:main}
Suppose that $h(j)\ge j+1$ for any $j\in [n-1]$.
Then $H^2_T(\Hess(S,h))$ is generated by $t_i,x_i$ $(i\in [n])$, $y_{j,k}$ $(j\in \J(h), k\in [n])$, and $\tau_A$ $(|A|\in \L(h))$.
\end{prop}

\begin{proof}
Recall that $\Gamma(h)$ is the labeled graph introduced at the end of Section~\ref{sect:3}.
As before, we consider the decomposition of $\Sn$:
\[
	\mathfrak{S}_n=\Sn^1\sqcup \Sn^2\sqcup\dots\sqcup \Sn^n\qquad (\Sn^j:=\{w\in \mathfrak{S}_n\mid w(j)=n\})
\]
and the Hessenberg function $h^j$ obtained by removing all the boxes in the $j$-th row and all the boxes in the $j$-th column from the configuration corresponding to $h$.
Note that the full subgraph of $\Gamma(h)$ with vertices $\Sn^j$ is $\Gamma(h^j)$.

We prove the proposition by induction on $n$ following the idea developed in \cite{fu-is-ma}.
The idea is also used in \cite{AHM}.
Let $z$ be an arbitrary element of $H^2_T(\Hess(S,h))$.

\smallskip
{\bf Step 1.}
Since $\Sn^n$ is isomorphic to $\mathfrak{S}_{n-1}$, it follows from the inductive assumption that $H^2_T(\Hess(S',h^n))$ is generated by $t_i$'s and elements corresponding to $x_\bullet,y_{\bullet},\tau_\bullet$, where $S'$ is a square matrix of size $n-1$ with distinct eigenvalues.
Indeed, those elements in $H^2_T(\Hess(S',h^n)$, denoted with $(n)$ as superscript, are defined as follows:
\begin{align*}
	x_{i}^{(n)}(w)&:=x_{i}(w)\quad (i\in [n-1])\\
	y_{j,k}^{(n)}(w)&:=y_{j,k}(w)\quad (j\in \J(h^n))\\
	\tau_B^{(n)}(w)&:=\tau_{B}(w)\quad (B\subset [n-1],\ |B|\in \L(h^n))
\end{align*}
where $w\in \Sn^n$.
This shows that any element of $H^2_T(\Hess(S',h^n))$ is the restriction image of a linear combination of $t_\bullet,x_\bullet,y_{\bullet},\tau_\bullet$ to $H^2_T(\Hess(S',h^n)$.
Therefore, we may assume that $z=0$ on $\Sn^n$ by subtracting an appropriate linear combination of $t_\bullet,x_\bullet,y_{\bullet},\tau_\bullet$ from $z$.

\smallskip
{\bf Step 2.} Suppose that $z=0$ on $\Sn^{r+1}\sqcup \cdots \sqcup \Sn^{n}$ for some $1\le r\le n-1$.
Then we shall show that $z$ minus an appropriate linear combination of $t_\bullet,x_\bullet,y_{\bullet},\tau_\bullet$ vanishes on $\Sn^{r+1}\sqcup \cdots \sqcup \Sn^n$.
We consider two cases.

\smallskip
Case 1.
The case where $h(r)\ge r+2$ (so $r < n-1$).
In this case, the $(r+1,r)$ and $(r+2,r)$ boxes are shaded in the configuration associated to $h$.
This means that at each vertex $w\in \Sn^r$, there are edges in $\Gamma(h)$ emanating from $w$ to $w\cdot(r,r+1)\in \Sn^{r+1}$ and $w\cdot(r,r+2)\in \Sn^{r+2}$, where the labels on those edges are respectively $t_{w(r)}-t_{w(r+1)}$ and $t_{w(r)}-t_{w(r+2)}$ up to sign.
Since $z=0$ on $\Sn^{r+1}\sqcup \cdots \sqcup\Sn^n$, $z(w)$ must be divisible by these linear polynomials.
However, since the cohomological degree of $z$ is two, this implies that $z$ must be $0$ on $\Sn^r$.

\smallskip
Case 2.
The case where $h(r)=r+1$.
In this case, each vertex $w$ in $\Sn^r$ is joined by an edge to a vertex $w\cdot(r,r+1)$ of $\Sn^{r+1}$.
Since the label on the edge is $t_{w(r)}-t_{w(r+1)}=t_n-t_{w(r+1)}$ up to sign and $z=0$ on $\Sn^{r+1}$, $z(w)$ for $w\in \Sn^r$ is a constant multiple of $t_n-t_{w(r+1)}$.
Note that $h(r-1)\le h(r)=r+1$ and
\[
	t_n-t_{w(r+1)}=
	\begin{cases}
		y_{r,n}(w)\quad &\text{(when $h(r-1)= r+1$)}\\
		\tau_{\{w(1),\dots,w(r)\}}(w)\quad &\text{(when $h(r-1)=r$)}.
	\end{cases}
\]

We distinguish two cases according to the above.

\smallskip
(i) The case where $h(r-1)= r+1$.
In this case, $h^r(j)\ge j+1$ for any $j\in [n-2]$, so the graph $\Gamma(h^r)$ is connected.
By the observation above, we have $z(w)=c_w y_{r,n}(w)$ $(w\in \Sn^r)$ with some integer $c_w$.

\smallskip
\noindent
{\it Claim.} $c_w=c_v$ for any $w,v\in \Sn^r$.

\smallskip
\noindent
{\it Proof of the claim.}
Since $\Gamma(h^r)$ is connected, it suffices to prove the identity for a pair of $w$ and $v$ joined by an edge of $\Gamma(h^r)$.
Then, $v=w\cdot(p,q)$ for some $p,q\in [n]\backslash\{r\}$ and
\begin{equation} \label{eq:zw-zv}
\begin{split}
z(w)-z(v)&=c_wy_{r,n}(w)-c_vy_{r,n}(v)\\
&=c_w(t_n-t_{w(r+1)})-c_v(t_n-t_{v(r+1)})\\
&=(c_w-c_v)t_n-c_wt_{w(r+1)}+c_vt_{v(r+1)},
\end{split}
\end{equation}
which must be divisible by $t_{w(p)}-t_{w(q)}=t_{v(q)}-t_{v(p)}$.
Here, $n=w(r)=v(r)$, $p\not=r$ and $q\not=r$, so the coefficient $c_w-c_v$ of $t_n$ in \eqref{eq:zw-zv} must be zero.
This proves the claim.

\smallskip
By the claim, we may write $c_w$ as $c$ and $z-cy_{r,n}=0$ on $\Sn^r$.
Furthermore, $y_{r,n}=0$ on $\Sn^{r+1}\sqcup \cdots \sqcup\Sn^{n}$ because $y_{r,n}(v)=0$ for $v\in \Sn$ with $n\in \{v(r+1),\dots,v(n)\}$ by the definition of $y_{r,n}$ in Lemma~\ref{lemm:xytau} (2).
Therefore, $z-cy_{r,n}=0$ on $\Sn^r\sqcup \cdots \sqcup\Sn^n$.

\smallskip
(ii) The case where $h(r-1)=r$.
In this case, the graph $\Gamma(h^r)$ is disconnected since $h^r(r-1)=r-1$.
Indeed, there are $\binom{n-1}{r-1}$ many connected components because $h^r(j)\ge j+1$ for any $1\le j\le n-2$ with $j\not=r-1$.
The vertex set of a connected component of $\Gamma(h^r)$ is
\[
	\Sn^r(A):=\{w\in \Sn^r\mid \{w(1),\dots,w(r)\}=A\}
\]
for some $n\in A\subset [n]$ with $|A|=r$.
On this connected component, $z(w)=c_A\tau_A(w)$ $(w\in \Sn^r(A))$ with some integer $c_A$, where $c_A$ is independent of $w\in \Sn^r(A)$ by a similar argument to the claim above, and $\tau_A$ vanishes on $\bigsqcup_{B\not=A}\Sn^r(B)=\Sn^r\backslash\Sn^r(A)$.
Therefore,
\begin{equation} \label{eq:z_minus}
	z-\sum_{n\in A\subset [n],\, |A|=r}c_A\tau_A=0 \quad\text{on $\Sn^r(A)$}.
\end{equation}
Furthermore, since $n\in A$ and $|A|=r$, $\tau_A$ in \eqref{eq:z_minus} vanishes on $\Sn^{r+1}\sqcup \cdots \sqcup\Sn^{n}$ by the definition of $\tau_A$ in Lemma~\ref{lemm:xytau}(3).
Therefore, the left hand side in \eqref{eq:z_minus} vanishes on $\Sn^r\sqcup \cdots \sqcup\Sn^n$.

The inductive argument developed above shows that one can change $z$ into $0$ on the whole set $\Sn$ by subtracting an appropriate linear combination of $t_\bullet,x_\bullet,y_{\bullet},\tau_\bullet$.
Since $z$ is an arbitrary element of $H^2_T(\Hess(S,h))$, this proves the proposition.
\end{proof}

\section{Structure of \texorpdfstring{$H^2(\Hess(S,h))$}{H2(Hess(S,h))}} \label{sect:5}

\subsection{Explicit presentation of $H^2(\Hess(S,h))$}
The generators of $H^2_T(\Hess(S,h))$ in Proposition~\ref{prop:main} can be decreased.
Indeed, it follows from Lemma~\ref{lemm:relation}(4) that we can drop $y_{n-1,k}$ when $n-1\in \J(h)$ and $\tau_A$ with $|A|=n-1$ when $n-1\in \L(h)$, i.e.~ it suffices to consider $\J(h)\backslash\{n-1\}$ and $\L(h)\backslash\{n-1\}$ as index sets of $y_{j,k}$ and $\tau_A$.
This is also true for the ordinary cohomology $H^2(\Hess(S,h))$ because
\[
	H^2(\Hess(S,h))=H^2_T(\Hess(S,h))/\Z\langle t_1,\dots,t_n\rangle
\]
where $\Z\langle t_1,\dots,t_n\rangle$ denotes the module generated by $t_1,\dots,t_n$ over $\Z$.
Thus, we have a surjective homomorphism
\[
	\Phi\colon \Z\langle X_i, Y_{j,k}, T_A\rangle \to H^2(\Hess(S,h))
\]
sending $X_i$ to $x_i$, $Y_{j,k}$ to $y_{j,k}$, and $T_A$ to $\tau_A$, where
\begin{enumerate}
\item $i\in [n]$,
\item $j\in \J(h)\backslash\{n-1\},\ k\in [n]$,
\item $A\subset [n],\ |A|\in \L(h)\backslash\{n-1\}$.
\end{enumerate}
By Lemma~\ref{lemm:relation}, the submodule $U$ of the free module $\Z\langle X_i, Y_{j,k}, T_A\rangle$ generated by the following elements maps to zero by $\Phi$:
\begin{enumerate}
\item[(R1)] $X_1+\cdots+X_n$,
\item[(R2)] $\displaystyle{\sum_{k=1}^nY_{j,k}-(X_1+\cdots+X_j-jX_{j+1})}$ \quad$(j\in \J(h)\backslash \{n-1\})$,
\item[(R3)] $\displaystyle{\sum_{|A|=j}T_A-(X_j-X_{j+1})}$\quad $(j\in \L(h)\backslash\{n-1\})$.
\end{enumerate}
Therefore, the map $\Phi$ induces a surjective homomorphism
\begin{equation} \label{eq:barPhi}
	\bar{\Phi}\colon \Z\langle X_i, Y_{j,k}, T_A\rangle/U\to H^2(\Hess(S,h)).
\end{equation}

With this understanding, we have the following.

\begin{theo} \label{theo:H2}
Suppose that $h(j)\ge j+1$ for any $j\in [n-1]$.
Then, the map $\bar{\Phi}$ in \eqref{eq:barPhi} is an isomorphism.
\end{theo}

\begin{proof}
First we note that the quotient module $\Z\langle X_i, Y_{j,k}, T_A\rangle/U$ is free.
Indeed, if $V$ is the submodule of $\Z\langle X_i, Y_{j,k}, T_A\rangle$ generated by
\begin{enumerate}
\item $X_i$\ $(i\not=n)$,
\item $Y_{j,k}$\ $(k\not=n)$,
\item $T_A$\ $(A\not=[|A|])$,
\end{enumerate}
then $\Z\langle X_i, Y_{j,k}, T_A\rangle=U\oplus V$.
Therefore, $\Z\langle X_i, Y_{j,k}, T_A\rangle/U\cong V$, which is free.

The rank of the free module $\Z\langle X_i, Y_{j,k}, T_A\rangle$ is
\[
	n+n(|\J(h)\backslash\{n-1\}|)+\sum_{j\in \L(h)\backslash\{n-1\}}\binom{n}{j}=n|\J(h)|+\sum_{j\in \L(h)}\binom{n}{j}
\]
while the elements in (R1), (R2), (R3) are linearly independent so that the rank of $U$ is
\[
	1+|\J(h)\backslash\{n-1\}|+|\L(h)\backslash\{n-1\}|=|\J(h)|+|\L(h)|.
\]
Therefore, the rank of the source module $\Z\langle X_i, Y_{j,k}, T_A\rangle/U$ of $\bar\Phi$ is
\[
	(n-1)|\J(h)|+\sum_{j\in \L(h)}\binom{n}{j}-|\L(h)|
\]
which agrees with the second Betti number of $\Hess(S,h)$ by Lemma~\ref{lemm:2_betti}.
This implies that $\bar\Phi$ is an isomorphism because $\bar\Phi$ is surjective and both the source and target modules of $\bar\Phi$ are free.
\end{proof}

Following Lemma~\ref{lemm:action_on_xytau}, we define an action of $\Sn$ on the variables $X_i, Y_{j,k}, T_A$ by
\begin{equation} \label{eq:action_on_XYT}
	\sigma\cdot X_i:=X_{i},\quad \sigma\cdot Y_{j,k}:=Y_{j,\sigma(k)},\quad \sigma\cdot T_A:=T_{\sigma(A)}
\end{equation}
and extend the action to the free module $\Z\langle X_i,Y_{j,k},T_A\rangle$ linearly.
Then, $\Sn$ acts on the submodule $U$ trivially, so that the $\Sn$-action on $\Z\langle X_i,Y_{j,k},T_A\rangle$ descends to an $\Sn$-action on the quotient $\Z\langle X_i,Y_{j,k},T_A\rangle/U$ and the isomorphism $\bar\Phi$ becomes $\Sn$-equivariant.

\subsection{$\Sn$-module structure on $H^2(\Hess(S,h))$}

Let $\lambda=(\lambda_1,\dots,\lambda_\ell)$ be a partition of $n$, denoted by $\lambda\vdash n$, and let $\mathfrak{S}_\lambda=\mathfrak{S}_{\lambda_1}\times\cdots\times\mathfrak{S}_{\lambda_\ell}$ the Young subgroup of $\Sn$ associated to $\lambda$.
We set
\[
	M^\lambda=\C[\Sn]\otimes_{\C[\mathfrak{S}_\lambda]}\C
\]
where $\C[G]$ denotes the group ring of a finite group $G$ over $\C$.
As is well-known, $\{M^\lambda\mid \lambda\vdash n\}$ forms an additive basis of the complex representation ring $R(\Sn)$ of $\Sn$.
Therefore, any $\Sn$-module over $\C$ can be expressed uniquely as a linear combination of $M^\lambda$'s over $\Z$.

\begin{theo}
Suppose that $h(j)\ge j+1$ and for $1\le j\le n-2$, let
\[
	\beta_j:=
	\begin{cases}
		(n-j,j) \quad&\text{if } h(j-1)=j,\ h(j)=j+1,\\
		(n-1,1)\quad&\text{if } h(j-1)=h(j)=j+1,\\
		(n)\quad&\text{otherwise},
	\end{cases}
\]
where $h(0)=1$.
Then
\[
	H^2(\Hess(S,h))\otimes\C=\sum_{j=1}^{n-2}M^{\beta_j}+ M^{(n)}\quad\text{in }R(\Sn).
\]
\end{theo}

\begin{proof}
Since the action of $\Sn$ on $X_i, Y_{j,k}, T_A$ is given by \eqref{eq:action_on_XYT}, $\C\langle X_i,Y_{j,k},T_A\rangle$ decomposes into a direct sum
\begin{align*}
&\C\langle X_i\mid i\in [n]\rangle\\
\oplus\ &\C\langle Y_{j,k}\mid j\in \J(h)\backslash\{n-1\},\ k\in [n]\rangle\\
\oplus\ &\C\langle T_A\mid A\subset [n],\ |A|\in \L(h)\backslash\{n-1\}\rangle.
\end{align*}
as an $\Sn$-module and the submodule $U\otimes \C$ of $\C\langle X_i,Y_{j,k},T_A\rangle$ is trivial as an $\Sn$-module.
The space $\C\langle Y_{j,k}\mid k\in [n]\rangle$ is isomorphic to $M^{(n-1,1)}$ for $j\in \J(h)$ while the space $\C\langle T_A\mid |A|=j\rangle$ is isomorphic to $M^{(n-j,j)}$ for $j\in\L(h)$.
There is no more non-trivial $\Sn$-module in $\C\langle X_i,Y_{j,k},T_A\rangle/U\otimes\C$ and one can see that the dimension of the complementary module is the number of $j\in [n-2]\backslash(\J(h)\sqcup\L(h))$ plus $1$ (this $1$ comes from $j=n-1$ and corresponds to $M^{(n)}$ in the last part of the right hand side in the theorem).
This proves the theorem.
\end{proof}

\begin{exam}[cf.~ Example~6.2 in \cite{CHL21}]
Let $n=8$ and $h=(2,3,6,6,6,7,8,8)$.
Then
\[
	\beta_1=(7,1),\ \beta_2=(6,2),\ \beta_3=(8),\ \beta_4=(8),\ \beta_5=(7,1),\ \beta_6=(2,6).
\]
Therefore,
\[
	H^2(\Hess(S,h))\otimes\C=3M^{(8)}+2M^{(7,1)}+2M^{(6,2)}\quad\text{in } R(\mathfrak{S}_8).
\]
\end{exam}

\section{A generalization} \label{sect:6}

In the previous sections, we studied $H^2(\Hess(S,h))$ under the condition that $h(j)\ge j+1$ for any $j\in [n-1]$.
In this section, we will study $H^{2d}(\Hess(S,h))$ under the condition that $h(j)\ge j+d$ for any $j\in [n-d]$, where $d\ge 2$.

For $j\in [n]$ and $k\in [n]$, there is an element $y_{j,k}\in H_T^{2(h(j)-j)}(\Hess(S,h))$ defined by
\begin{equation} \label{eq:yjkd}
	y_{j,k}(w):=
	\begin{cases}
		\prod_{\ell=j+1}^{h(j)}(t_k-t_{w(\ell)})\quad&(\text{if }k\in \{w(1),\dots,w(j)\})\\
		0\quad&(\text{otherwise}).
	\end{cases}
\end{equation}
The element $y_{j,k}$ is introduced in \cite{AHHM} and proved to be in $H^*_T(\Hess(S,h))$ (\cite[Lemma~10.2]{AHHM}).
Note that when $d=1$ and $h(j)=j+1$, the element $y_{j,k}$ above agrees with the $y_{j,k}$ in Lemma~\ref{lemm:xytau}(2) (see also Remark~\ref{rema:4-1}(2)).
We use the same notation $y_{j,k}$ for its image in the ordinary cohomology $H^{2d}(\Hess(S,h))$.
The dot action on $y_{j,k}$ is the same as before, i.e.~
\begin{equation} \label{eq:sigma_yjk}
	\sigma\cdot y_{j,k}=y_{j,\sigma(k)}\quad \text{for $\sigma\in\Sn$}.
\end{equation}
Therefore, $\sum_{k=1}^ny_{j,k}$ is $\Sn$-invariant.
In fact, the sum has the following expression:
\begin{equation} \label{eq:yjksum}
	\sum_{k=1}^ny_{j,k}=\sum_{i=1}^j\prod_{\ell=j+1}^{h(j)}(x_{i}-x_{\ell}),
\end{equation}
which can be proved by checking that both sides take the same value at each $w\in \Sn$.

We set
\[
	\D_d(h):=\{j\in [n-d]\mid h(j)=j+d\}.
\]
Our main result in this section is the following.

\begin{theo} \label{theo:H2d}
Suppose that $d\ge 2$ and $h(j)\ge j+d$ for any $j\in [n-d]$.
Then, the restriction map
\[
	\iota^*\colon H^{2p}(\flag(n))\to H^{2p}(\Hess(S,h))
\]
is an isomorphism for $p<d$.
For $p=d$, the restriction map is an injective and we have an isomorphism
\begin{align*}
	&H^{2d}(\Hess(S,h))\\
	\cong&\left(\iota^*( H^{2d}(\flag(n)))\oplus\Z\langle Y_{j,k}\mid j\in \D_d(h), k\in [n]\rangle\right)/\Z\langle \sum_{k=1}^nY_{j,k}-y_j\mid j\in \D_d(h)\rangle
\end{align*}
where $Y_{j,k}$ corresponds to $y_{j,k}$ in $H^*(\Hess(S,h))$ and $\displaystyle{y_j=\sum_{i=1}^j\prod_{\ell=j+1}^{h(j)}(x_{i}-x_{\ell})}$.
\end{theo}

\begin{rema}
(1) Theorem~\ref{theo:MPS}(3) says that $\Hess(S,h)$ is connected, in other words, the restriction map
\[
	\iota\colon H^0(\flag(n))\to H^0(\Hess(S,h))
\]
is an isomorphism if and only if $h(j)\ge j+1$ for any $j\in [n-1]$.
Therefore, Theorem~\ref{theo:H2d} can be regarded as a generalization of Theorem~\ref{theo:MPS}(3) and Theorem~\ref{theo:H2}.

(2) When $d=1$, the elements $\tau_A$ appear in $H^2(\Hess(S,h))$ as shown in Theorem~\ref{theo:H2} but such type of elements does not appear when $d\ge 2$.
\end{rema}

Since the action of $\Sn$ on $x_i$'s is trivial, so is that on $\iota^*(H^{2d}(\flag(n)))$ while that on $y_{j,k}$ is given by \eqref{eq:sigma_yjk}.
Therefore, we obtain the following corollary from Theorem~\ref{theo:H2d}.

\begin{coro}
Let the situation be as in Theorem~\ref{theo:H2d}.
Then, the action of $\Sn$ on $H^{2p}(\Hess(S,h))$ is trivial while
\[
	H^{2d}(\Hess(S,h))\otimes\C=m_d M^{(n)}+|\D_d(h)|M^{(n-1,1)}\quad \text{in $R(\Sn)$}
\]
where $m_d=b_{2d}(\flag(n))-|\D_d(h)|$.
\end{coro}

\begin{rema}
Since
\[
	\Poin(\flag(n),\sqrt{q})=\prod_{i=1}^n\frac{1-q^i}{1-q},
\]
we have $b_{2d}(\flag(n))\ge n-1$ while $|\D_d(h)|\le n-2$.
Therefore, $m_d$ in the corollary is positive.
\end{rema}

The idea of the proof of Theorem~\ref{theo:H2d} is the same as before.
We compute the Betti numbers of $\Hess(S,h)$ up to degree $2d$ using Proposition~\ref{prop:inductive_formula} and observe that $H^{*}(\Hess(S,h))$ is generated by $x_i$'s and $y_{j,k}$'s as a graded ring up to degree $2d$.

For the first part of Theorem~\ref{theo:H2d}, we have a homotopical version which is proved easier and may be of independent interest.

\begin{prop}\label{propHomotopy}
Suppose that $d\ge 1$ and $h(j)\ge j+d$ for any $j\in [n-d]$.
Then, the natural inclusion $\iota\colon \Hess(S,h)\to \flag(n)$ induces isomorphisms of homotopy groups $\pi_q$ in degrees $q\leq 2d-1$.
\end{prop}

To prove this, it is sufficient to construct cellular structures on $\Hess(S,h)$ and $\flag(n)$ which are consistent under $\iota$ and coincide in small dimensions. Instead of cellular structures, one can use affine pavings. There is an affine paving of $\flag(n)$ by even-dimensional Bruhat cells indexed by permutations $w\in\Sn$. Intersecting a Bruhat cell $C_w$ of $\flag(n)$ with the subvariety $\Hess(S,h)$ gives an affine cell $C'_w$ of $\Hess(S,h)$. The dimensions of cells $C_w$ and $C_w'$ can be computed from Bialynicki-Birula theory, see~\cite{ma-pr-sh}. The dimension $\dim_\C C_w$ equals the number $\ell(w)$ of inversions of $w$, while $\dim_\C C_w'$ equals $\ell_h(w)$, the number of inversions $j<i$, $w(j)>w(i)$, satisfying $i\leq h(j)$, see~\eqref{eqDhw}.

\begin{lemm}\label{lemInversions}
Let $h(j)\ge j+d$ for any $j\in [n-d]$. Then for any permutation $w$, each one of the conditions $\ell_h(w)<d$ or $\ell(w)<d$ implies $\ell_h(w)=\ell(w)$.
\end{lemm}

\begin{proof}
Since $\ell_h(w)\le \ell(w)$, it suffices to prove $\ell_h(w)=\ell(w)$ when $\ell_h(w)<d$.  
Suppose that $\ell_h(w)\not=\ell(w)$.  Then there is an inversion $\{j,i\}$ $(j<i)$ in $w$ which contributes to $\ell(w)$ but does not contribute to $\ell_h(w)$. This means $j+d\le h(j)< i$. 

We assume that the difference $i-j(>d)$ is minimum among those inversions.  Any number $j'$ such that $j<j'\leq j+d$ either produces an inversion $\{j,j'\}$ or $\{j',i\}$. Obviously $j'-j\le d$ and if $\{j',i\}$ is an inversion, then $i-j'\le d$ which follows from the minimality of $i-j$.    
In any case, each $j'$ produces an inversion which contributes to $\ell_h(w)$. Since there are $d$ many such $j'$,  we have $\ell_h(w)\ge d$.  However, this contradicts the condition $\ell_h(w)<d$.  Therefore $\ell_h(w)=\ell(w)$. 
\end{proof}

Let us prove Proposition~\ref{propHomotopy}.

\begin{proof}
Since the cell $C_w'$ of $\Hess(S,h)$ is the intersection of the Bruhat cell $C_w$ of $\flag(n)$ with $\Hess(S,h)$, Lemma~\ref{lemInversions} implies that the spaces $\flag(n)$ and $\Hess(S,h)$ have the same $(2d-1)$-skeleta and that any $2d$-dimensional cell $C_w$ of $\flag(n)$ agrees with the cell $C_w'$ of $\Hess(S,h)$. Therefore 
\[
\iota_*\colon H_q(\Hess(S,h))\to H_q(\flag(n))
\]
an isomorphism for $q\le 2d-1$ and an epimorphism for $q=2d$.
Moreover, both $\Hess(S,h)$ and $\flag(n)$ are simply connected. Therefore the proposition follows from the Whitehead theorem (\cite[Theorem in p.399]{span66}).  
\end{proof} 


We now proceed with the more detailed analysis of homology needed to prove Theorem~\ref{theo:H2d}.

\subsection{Betti numbers}
The Betti numbers $b_{2i}(\Hess(S,h))$ of $\Hess(S,h)$ for $i\le d$ are given as follows.

\begin{lemm} \label{lemm:6-1}
Suppose that $d\ge 2$ and $h(j)\ge j+d$ for any $j\in [n-d]$.
Then
\[
	b_{2i}(\Hess(S,h))=
	\begin{cases}
		b_{2i}(\flag(n))\quad&(\text{if }i<d)\\
		b_{2i}(\flag(n))+(n-1)|\D_d(h)|\quad&(\text{if }i=d).
	\end{cases}
\]
\end{lemm}

\begin{proof}
For a polynomial $f(q)$ in $q$, we denote by $f(q)^{\le d}$ the polynomial obtained from $f(q)$ by truncating terms of degree $>d$.
Then, the lemma is equivalent to
\begin{equation} \label{eq:restatement}
	\Poin(\Hess(S,h),\sqrt{q})^{\le d}=\Poin(\flag(n),\sqrt{q})^{\le d}+(n-1)|\D_d(h)|q^d.
\end{equation}
We prove the identity \eqref{eq:restatement} by induction on $n+d$ where $d\ge 2$.
Since $d\ge 2$ and $h(j)\ge j+d$ for any $j\in [n-d]$ by assumption, $n$ is greater than or equal to $3$ and when $(n,d)=(3,2)$, $h$ must be $(3,3,3)$.
In this case, $\Hess(S,h)=\flag(3)$, $\D_d(h)=\emptyset$ and hence lemma holds.

Suppose that $n+d\ge 6$ and the lemma holds for any pair $(n',d')$ such that $n'+d'<n+d$.
It follows from Proposition~\ref{prop:inductive_formula} that we have
\begin{equation} \label{eq:led}
\begin{split}
	\Poin(\Hess(S,h),\sqrt{q})^{\le d}
	=&\sum_{j=1}^{n-d-1}\left(q^{h(j)-j}\Poin(\Hess(S',h^j),\sqrt{q})\right)^{\le d}\\
	&+\sum_{j=n-d}^n\left(q^{h(j)-j}\Poin(\Hess(S',h^j),\sqrt{q})\right)^{\le d}\\
	=|\D_d(h)|q^d+\sum_{j=n-d}^n&\left(q^{n-j}\Poin(\Hess(S',h^j),\sqrt{q})\right)^{\le d}
\end{split}
\end{equation}
because $h(j)\ge j+d$ for any $j\in [n-d]$ and $h(j)=j+d<n$ if and only if $j\in \D_d(h)$.
Applying the induction assumption (and Lemma~\ref{lemm:2_betti} when $d=2$) to the last sum in \eqref{eq:led}, we obtain
\begin{equation} \label{eq:ledd}
\begin{split}
	&\sum_{j=n-d}^n\left(q^{n-j}\Poin(\Hess(S',h^j),\sqrt{q})\right)^{\le d}\\
	=&\sum_{j=n-d}^{n-2}\left(q^{n-j}\Poin(\flag(n-1),\sqrt{q})\right)^{\le d}\\
	&+\left(q\Big(\Poin(\flag({n-1}),\sqrt{q})+(n-2)|\D_{d-1}(h^{n-1})|q^{d-1}\Big)\right)^{\le d}\\
	&+\left(\Poin(\flag(n-1),\sqrt{q})+(n-2)|\D_d(h^n)|q^d\right)^{\le d}\\
	=&\sum_{j=n-d}^n\left(q^{n-j}\Poin(\flag(n-1),\sqrt{q})\right)^{\le d}\\
	&+(n-2)\left(|\D_{d-1}(h^{n-1})|+|\D_d(h^n)|\right)q^d\\
	=&\Poin(\flag(n),\sqrt{q})^{\le d}+(n-2)\left(|\D_{d-1}(h^{n-1})|+|\D_d(h^n)|\right)q^d,
\end{split}
\end{equation}
where we can see the last identity above by applying Proposition~\ref{prop:inductive_formula} to the flag variety $\flag(n)$.
Thus, if we prove
\begin{equation} \label{eq:Ddh}
	|\D_{d-1}(h^{n-1})|+|\D_d(h^n)|=|\D_d(h)|,
\end{equation}
then the identity \eqref{eq:restatement} follows from \eqref{eq:led} and \eqref{eq:ledd}.
However, one can easily see that
\[
	(|\D_{d-1}(h^{n-1})|,|\D_d(h^n)|)
	=\begin{cases}
		(0,|\D_d(h)|) \quad&(\text{if }h(n-d-1)=n)\\
		(1,|\D_d(h)|-1)\quad&(\text{if }h(n-d-1)=n-1)
	\end{cases}
\]
and this implies \eqref{eq:Ddh}.
\end{proof}

\subsection{Complementary elements}
We introduce complementary elements $y^*_{i,k}$ which will make our argument perspective.

The element $y_{j,k}$ in \eqref{eq:yjkd} is defined by looking at the $j$-th column of the configuration associated to the Hessenberg function $h$.
Similarly, one can define an element $y^*_{i,k}$ of $H^*_T(\Hess(S,h))$ by looking at the $i$-th row of the configuration as follows.
For $1<i\le n$, we define
\[
	h^*(i):=\min\{j\in [n]\mid h(j)\ge i\}.
\]
This definition tells us that the shaded boxes in the $i$-th row and under the diagonal in the configuration associated to $h$ are at positions $(i,\ell)$ $(h^*(i)\le \ell <i)$.
Looking at those shaded boxes, we define
\begin{equation} \label{eq:y*ik}
	y^*_{i,k}(w):=
	\begin{cases}
		\prod_{\ell=h^*(i)}^{i-1}(t_k-t_{w(\ell)})\quad&(k\in \{w(i),\dots,w(n)\})\\
		0\quad&(\text{otherwise}).
	\end{cases}
\end{equation}
One can see that $y^*_{i.k}$ is in $H^*_T(\Hess(S,h))$ similarly to $y_{j,k}$.

Note that
\begin{enumerate}
\item $h(j)\ge j+d$ for $\forall j\in [n-d]$ $\Longleftrightarrow$ $h^*(i)\le i-d$ for $d+1\le\forall i\le n$.
\item Under the assumption that $h(j)\ge j+d$ for any $j\in [n-d]$, we have
\[
	h(j)=j+d<n \ (\text{i.e.}\ j\in \D_d(h)) \Longleftrightarrow h^*(j+1+d)=j+1.
\]
\end{enumerate}
Based on this observation, we define
\[
	\D^*(h)=\{ i\mid d+2\le i\le n,\ h^*(i)=i-d\}.
\]
Then, (2) above can be restated that $j\in \D_d(h)\Longleftrightarrow j+1+d\in \D^*_d(h)$.

The elements $y^*_{i,k}$ $(i\in \D^*_d(h))$ do not provide new elements as is seen from the following lemma.

\begin{lemm} \label{lemm:y+y*}
Suppose that $h(j)\ge j+d$ for any $j\in [n-d]$.
Then, for $j\in \D_d(h)$ we have
\[
	y_{j,k}+y^*_{j+1+d,k}=\prod_{\ell=j+1}^{j+d}(t_k-x_\ell).
\]
\end{lemm}

\begin{proof}
It immediately follows from the definitions of $y_{j,k}$, $y^*_{i,k}$, and $x_i$ that the both sides in the lemma take the same value at every $w\in \Sn$.
\end{proof}

\subsection{Generators of $H^*_T(\Hess(S,h))$}
We show that the elements $t_i,x_i,y_{j,k}$ generate $H^*_T(\Hess(S,h))$ up to $*\le 2d$ as a graded ring under our assumption.

\begin{lemm} \label{lemm:6-2}
Suppose that $h(j)\ge j+d$ for any $j\in [n-d]$ and $d\ge 2$.
Then $H_T^{*}(\Hess(S,h))$ is generated by $t_i$, $x_i$ $(i\in [n])$ and $y_{j,k}$ $(j\in \D_d(h), k\in [n])$ up to $*\le 2d$ as a graded ring, where the degree of $x_i$ is $2$ while that of $y_{j,k}$ is $2d$.
\end{lemm}

\begin{proof}
We prove the lemma by induction on $n+d$ similarly to the proof of Lemma~\ref{lemm:6-1}.
When $(n,d)=(3,2)$, $\Hess(S,h)=\flag(3)$ and the lemma holds since $H_T^*(\flag(n))$ is generated by $t_i,x_i$ $(i\in [n])$.

Suppose that $n+d\ge 6$ and the lemma holds for any pair $(n',d')$ such that $n'+d'<n+d$.
Similarly to the proof of Proposition~\ref{prop:main}, we consider the decomposition
\[
	\Sn=\Sn^1\sqcup \Sn^2\sqcup\cdots\sqcup \Sn^n\qquad (\Sn^j=\{w\in\Sn\mid w(j)=n\}).
\]
Let $z$ be an arbitrary element of $H^{2p}_T(\Hess(S,h))$ for $p\le d$.

{\bf Step 1.} By the same reasoning as Step 1 in the proof of Proposition~\ref{prop:main}, we may assume $z=0$ on $\Sn^n$ by subtracting an appropriate polynomial in $t_\bullet, x_\bullet, y_{\bullet}$.

{\bf Step 2.} Suppose that $z=0$ on $\Sn^n$.
We will show that $z$ minus an appropriate polynomial of $t_\bullet, x_\bullet, y_\bullet$ vanishes on $\Sn^{n-1}\sqcup\Sn^n$.

Since $h(n-1)=n$, each vertex $w\in \Sn^{n-1}$ is connected to the vertex $w\cdot(n,n-1)$ of $\Sn^n$ by an edge of the labeled graph $\Gamma(h)$.
Since $z=0$ on $\Sn^n$, $z(w)$ must be divisible by the label $t_{w(n)}-t_{w(n-1)}=t_{w(n)}-t_{n}$ on the edge.
Therefore, there is a homogeneous element $g(w)\in \Z[t_1,\dots,t_n]$ of degree $2(p-1)$ such that
\begin{equation} \label{eq:zwr}
	z(w)=(t_{w(n)}-t_{n})g(w)\quad\text{for } w\in \Sn^{n-1}.
\end{equation}
We express
\begin{equation} \label{eq:gwell}
	g(w)=\sum_{\ell=0}^{p-1}g_\ell(w) t_n^\ell
\end{equation}
with homogeneous polynomial $g_\ell(w)$ in $\Z[t_1,\dots,t_{n-1}]$ of degree $2(p-1-\ell)$.

\smallskip
\noindent
{\bf Claim.}
If $v,w\in \Sn^{n-1}$ are joined by an edge of the labeled graph $\Gamma(h^{n-1})$, i.e.~ 
$v=w\cdot(i,j)$ for some transposition $(i,j)$ with $j<i\le h(j)$, $j\not=n-1$ and $i\not=n-1$, then
\[
	g_\ell(v)\equiv g_\ell(w) \mod{t_{w(i)}-t_{w(j)}}.
\]

\smallskip
\noindent
{\it Proof of Claim}.
Since $z$ is an element of $H^*_T(\Hess(S,h))$, it satisfies the congruence relation
\begin{equation} \label{eq:zvw}
	z(v)\equiv z(w)\mod{t_{w(i)}-t_{w(j)}}.
\end{equation}
Since $v=w\cdot(i,j)$, we have $v(i)=w(j)$, $v(j)=w(i)$ and $v(s)=w(s)$ for $s\not=i,j$.
Moreover, $w(i)$ and $w(j)$ are not equal to $n$ because $i$ and $j$ are not equal to $n-1$ and $w\in \Sn^{n-1}$.
Therefore,
\[
	t_{v(n)}-t_{n}\equiv t_{w(n)}-t_{n}\not\equiv 0\mod{t_{w(i)}-t_{w(j)}}.
\]
This together with \eqref{eq:zwr}, \eqref{eq:gwell}, and \eqref{eq:zvw} implies the congruence relation in the claim.

\smallskip
By the claim above, each $g_\ell$ is an element of $H^*_T(\Hess(S',h^{n-1}))$ by Proposition~\ref{prop:GKM}.
Since $\ell\ge 0$ and $p\le d$, we have
\begin{equation} \label{eq:gell_degree}
	\text{the degree of $g_\ell$}=2(p-1-\ell)\le 2(d-1).
\end{equation}
We note that $h(n-d-1)=n$ or $n-1$ by the assumption $h(j)\ge j+d$ for any $j\in [n-d]$.
Now we take two cases according to the value of $h(n-d-1)$.

\smallskip
Case 1. The case where $h(n-d-1)=n$.
In this case, $h^{n-1}(j)\ge j+d$ for any $j\in [n-d-1]$.
Therefore, by the induction assumption and \eqref{eq:gell_degree}, any $g_\ell$ can be written as a polynomial in $t_i$, $x_i^{(n-1)}$ $(i\in [n-1])$, where for $w\in \Sn^{n-1}$,
\[
	x_i^{(n-1)}(w)=
	\begin{cases}
		x_i(w)\quad &(i\le n-2)\\
		x_{n}(w)\quad &(i=n-1).
	\end{cases}
\]
This shows that there is a polynomial $G_\ell$ in $t_i$, $x_i$'s whose restriction to $\Sn^{n-1}$ agrees with $g_\ell$.
Therefore, it follows from \eqref{eq:zwr} that
\[
	z=(x_n-t_n)\sum_{\ell=0}^{p-1}G_\ell t_n^\ell \quad\text{on }\Sn^{n-1}.
\]
Both sides above vanish on $\Sn^n$, so they agree on $\Sn^{n-1}\sqcup \Sn^n$.
Therefore, subtracting the right hand side above from $z$, we may assume $z=0$ on $\Sn^{n-1}\sqcup \Sn^n$.

\smallskip
Case 2.
The case where $h(n-d-1)=n-1$.
In this case, $h^{n-1}(j)\ge j+(d-1)$ for any $j\in [n-1-(d-1)]$ and $\D_{d-1}(h^{n-1})=\{n-d-1\}$.
Therefore, by the induction assumption and \eqref{eq:gell_degree}, any $g_\ell$ can be written as a polynomial in $t_i, x_i^{(n-1)}$ and $y^{(n-1)}_{n-d-1,k}$, where $k\in [n-1]$ and
\begin{equation*} \label{eq:ynd1}
	y_{n-d-1,k}^{(n-1)}(w)=y_{n-d-1,k}(w)/(t_k-t_{w(n-1)})\quad\text{ for $w\in \Sn^{n-1}$}.
\end{equation*}
By Lemma~\ref{lemm:y+y*}, we may use $y_{n-1,k}^{*(n-1)}$ instead of $y_{n-d-1,k}^{(n-1)}$, where
\begin{equation*} \label{eq:y*n-1}
	y_{n-1,k}^{*(n-1)}=y^*_{n,k}/(t_k-t_{w(n-1)}).
\end{equation*}
We note that since
\[
	\deg g_\ell=2(p-1-\ell),\quad \deg y_{n-1,k}^{*(n-1)}=2(d-1),\quad p\le d,
\]
$y_{n-1,k}^{*(n-1)}$ does not appear in the polynomial expression of $g_\ell$ unless $p=d$ and $\ell=0$.
Therefore, it follows from \eqref{eq:zwr} and \eqref{eq:gwell} that we can write
\begin{equation} \label{eq:zw2}
\begin{split}
	z(w)&=(t_{w(n)}-t_n)\left(\sum_{k=1}^{n-1}c_ky_{n-1,k}^{*(n-1)}(w)+\sum_{\ell=0}^{p-1}f_\ell(w)t_n^\ell\right)\\
	&=\sum_{k=1}^{n-1}c_k(t_{w(n)}-t_n)y_{n-1,k}^{*(n-1)}(w)+(t_{w(n)}-t_n)\sum_{\ell=0}^{p-1}f_\ell(w)t_n^\ell
\end{split}
\end{equation}
where $c_k\in \Z$ and $f_\ell$ is a polynomial in $t_i, x_i^{(n-1)}$'s $(i\in [n-1])$.

Similarly to Case 1, $f_\ell\in H^*(\Hess(S',h^{n-1}))$ in \eqref{eq:zw2} is the image of some $F_\ell\in H^*(\Hess(S,h))$.
Although $y_{n-1,k}^{*(n-1)}$ may not be in the image of the restriction map, we have
\begin{align*}
	(t_{w(n)}-t_n)y_{n-1,k}^{*(n-1)}(w)&=
	\begin{cases}
		\prod_{\ell=h^*(n)}^{n}(t_k-t_{w(\ell)})\quad &(k=w(n))\\
		0\quad&(\text{otherwise})
	\end{cases}\\
	&=y^*_{n,k}(w)\quad
\end{align*}
for $w\in\Sn^{n-1}$ (so $t_n=t_{w(n-1)}$), where $h^*(n)=n-d$ because $(n-d-1)=n-1$ and $h(j)=n$ for $j\ge n-d$.
This observation and \eqref{eq:zw2} show that
\begin{equation} \label{eq:z2}
	z=\sum_{k=1}^{n-1}c_ky^*_{n,k}+(x_n-t_n)\sum_{\ell=0}^{p-1}F_\ell t_n^\ell \quad \text{on $\Sn^{n-1}$}.
\end{equation}
Here, $z=0$ on $\Sn^n$ by assumption and the right hand side above also vanishes on $\Sn^n$.
Indeed, since $w(n)=n$ for $w\in \Sn^n$, we have $y^*_{n,k}(w)=0$ for $k\not=n$ and $(x_n-t_n)(w)=t_{w(n)}-t_n=0$.
Thus, the identity \eqref{eq:z2} holds on $\Sn^{n-1}\sqcup \Sn^n$.
This together with Lemm~\ref{lemm:y+y*} shows that $z$ minus an appropriate polynomial in $t_\bullet, x_\bullet, y_\bullet$ vanishes on $\Sn^{n-1}\sqcup \Sn^n$.

{\bf Step 3.} Suppose that $z=0$ on $\Sn^{r+1}\sqcup\cdots\sqcup\Sn^{n}$ for some $r$ with $n-d\le r\le n-2$.
Then, since $h(r)=n$, there are shaded boxes at positions $(r+1,r), (r+2,r), \dots,(n,r)$ in the configuration associated to $h$.
This means that each vertex $w\in \Sn^r$ is connected to $\Sn^\ell$ for $r+1\le \ell\le n$ by an edge with label $t_{w(\ell)}-t_{w(r)}=t_{w(\ell)}-t_n$.
Since $z=0$ on $\Sn^{r+1}\sqcup\cdots\sqcup\Sn^{n}$, $z(w)$ for $w\in \Sn^r$ must be divisible by $\prod_{\ell=r+1}^{n}(t_{w(\ell)}-t_n)$.
Therefore, there is a homogeneous element $g(w)\in \Z[t_1,\dots,t_n]$ such that
\begin{equation} \label{eq:zr}
z(w)=\left(\prod_{\ell=r+1}^{n}(t_{w(\ell)}-t_n)\right)g(w)\quad\text{for $w\in \Sn^r$}.
\end{equation}
The same argument as in the claim in Step 2 shows that the $g$ in \eqref{eq:zr} satisfies the congruence relation for the labeled graph $\Gamma(h^r)$.
Since $p\le d$ and $n-r\ge 2$, we have
\[
	\text{the degree of $g$}=2(p-n+r)\le 2(d-2).
\]
Moreover, $h^r(j)\ge j+(d-1)$ for any $j\in [n-1-(d-1)]$.
Therefore, by the induction assumption, $g$ can be expressed as a polynomial in $t_i$ and $x_i^{(r)}$ where
\[
	x_i^{(r)}(w):=\begin{cases} x_i(w) \quad&(i<r)\\
	x_{i+1}(w)\quad&(r\le i)\end{cases}
\]
for $w\in \Sn^r$.
Since $x_i^{(r)}$ is in the image of the restriction map from $H^*(\Hess(S,h))$ to $H^*(\Hess(S',h))$, there is an element $G\in H^*(\Hess(S,h))$ whose restriction to $H^*(\Hess(S',h^r))$ agrees with $g$.
It follows from \eqref{eq:zr} that
\begin{equation} \label{eq:zrr}
	z=\left(\prod_{\ell=r+1}^{n}(x_\ell-t_n)\right)G\quad \text{on } \Sn^r.
\end{equation}
Here, $z=0$ on $\Sn^{r+1}\sqcup\cdots\sqcup\Sn^{n}$ by assumption and the right hand side above also vanishes on $\Sn^{r+1}\sqcup\cdots\sqcup\Sn^{n}$.
Indeed, since $x_\ell(w)=t_{w(\ell)}=t_n$ for $w\in \Sn^\ell$, $\prod_{\ell=r+1}^{n}(x_\ell-t_n)(w)=0$ for $w\in \Sn^{r+1}\sqcup\cdots\sqcup\Sn^{n}$.
Thus, the identity \eqref{eq:zrr} holds on $\Sn^r\sqcup\cdots\sqcup\Sn^{n}$, so that $z$ minus an appropriate polynomial in $t_\bullet, x_\bullet, y_\bullet$ vanishes on $\Sn^r\sqcup \cdots \sqcup \Sn^n$.

{\bf Step 4.}
Suppose that $z=0$ on $\Sn^{r+1}\sqcup\cdots\sqcup\Sn^{n}$ for some $r$ with $1\le r\le n-d-1$.
Then, similarly to Step 3, $z(w)$ for $w\in \Sn^r$ must be divisible by $\prod_{\ell=r+1}^{h(r)}(t_n-t_{w(\ell)})$.
Here, the degree of $z$ is $2p\le 2d$ and $h(r)-r\ge d$, so $z(w)=0$ unless $p=d$ and $h(r)=r+d$.
When $p=d$ and $h(r)=r+d$, we have
\begin{equation} \label{eq:z4}
z(w)=c\left(\prod_{\ell=r+1}^{r+d}(t_n-t_{w(\ell)})\right)\quad \text{for } w\in \Sn^r,
\end{equation}
where $c\in \Z$.

Since $r\le n-d-1$ by assumption, $h(r)=r+d<n$.
Therefore, $r\in \D_d(h)$ so that we have an element $y_{r,k}\in H^{2d}_T(\Hess(S,h))$ for any $k\in [n]$.
We take $k=n$.
Since $h(r)=r+d$, we have
\[
	y_{r,n}(w)=\begin{cases}\prod_{\ell=r+1}^{r+d}(t_n-t_{w(\ell)})\quad&(n\in \{w(1),\dots,w(r)\})\\
	0\quad&(\text{otherwise})\end{cases}
\]
by definition.
This together with \eqref{eq:z4} shows that
\[
	z=cy_{r,n} \quad\text{on } \Sn^r.
\]
Here, $z=0$ on $\Sn^{r+1}\sqcup\cdots\sqcup\Sn^n$ by assumption and the right hand side above also vanishes because $y_{r,n}(w)=0$ if $n\in \{w(r+1),\dots,w(n)\}$ by definition.
Therefore, $z$ minus $cy_{r,n}$ vanishes on $\Sn^r\sqcup\cdots\sqcup\Sn^n$.

This completes the induction step and the lemma has been proven.
\end{proof}

\subsection{Proof of Theorem~\ref{theo:H2d}}
Under these preparations, we prove Theorem~\ref{theo:H2d}.
When $p<d$, the restriction map
\[
	\iota^*\colon H^{2p}(\flag(n))\to H^{2p}(\Hess(S,h))
\]
is surjective by Lemma~\ref{lemm:6-2} and indeed an isomorphism by Lemma~\ref{lemInversions} (or Lemma~\ref{lemm:6-1}).

When $p=d$, we consider the homomorphism
\[
	\Phi\colon \iota^*(H^{2d}(\flag(n)))\oplus\Z\langle Y_{j,k}\mid j\in\D_d(h), k\in [n]\rangle\to H^{2d}(\Hess(S,h))
\]
sending $Y_{j,k}$ to $y_{j,k}$.
The map $\Phi$ is surjective by Lemma~\ref{lemm:6-2} and $\sum_{k=1}^nY_{j,k}-y_j$ for $j\in \D_d(h)$ are in the kernel of $\Phi$ by \eqref{eq:yjksum}.
Therefore, the map $\Phi$ induces a surjective homomorphism
\begin{align*}
	\bar{\Phi}\colon \Big(\iota^*&(H^{2d}(\flag(n)))\oplus\Z\langle Y_{j,k}\mid j\in\D_d(h), k\in [n]\rangle\Big)/\Z\langle \sum_{k=1}^nY_{j,k}-y_j \rangle\\
	&\to H^{2d}(\Hess(S,h)).
\end{align*}
Here, the rank of the source module is at most the rank of the target module by Lemma~\ref{lemm:6-1}.
Moreover, one can easily see that the source module is torsion free and we know that $H^*(\Hess(S,h))$ is also torsion free.
Thus, the surjective homomorphism $\bar\Phi$ must be an isomorphism, proving the theorem.

\medskip
\noindent
{\bf Acknowledgement.} The work of first and second authors was done within the framework of the HSE University Basic Research Program.

\end{document}